\documentclass[11pt]{article}\UseRawInputEncoding
\usepackage{amssymb,amsfonts,amsmath,latexsym,epsf,tikz,url}
\usepackage[usenames,dvipsnames]{pstricks}
\usepackage{pstricks-add}
\usepackage{epsfig}
\usepackage{pst-grad} % For gradients
\usepackage{pst-plot} % For axes
\usepackage[space]{grffile} % For spaces in paths
\usepackage{etoolbox} % For spaces in paths
\makeatletter % For spaces in paths
\patchcmd\Gread@eps{\@inputcheck#1 }{\@inputcheck"#1"\relax}{}{}
\makeatother

\newtheorem{theorem}{Theorem}[section]

\newtheorem{corollary}[theorem]{Corollary}

\newtheorem{remark}[theorem]{Remark}
\newtheorem{example}[theorem]{Example}

\newcommand{\qed}{\hfill $\square$\medskip}

\begin{document}

\title{Some results on the super domination number of a graph}

\author{
Nima Ghanbari
}

\date{\today}

\maketitle

\begin{center}
Department of Informatics, University of Bergen, P.O. Box 7803, 5020 Bergen, Norway\\
{\tt   Nima.Ghanbari@uib.no}
\end{center}

%%%%%%%%%%%%%%ABSTRACT%%%%%%%%%%%%%%%%%%%%%%%%%%%%%%%%%%%%%%%%%%%%%%%%%%%%%%%%%%%%

\begin{abstract}
Let $G=(V,E)$ be a simple graph. A dominating set of $G$ is a subset $S\subseteq V$ such that every vertex not in $S$ is adjacent to at least one vertex in $S$.
The cardinality of a smallest dominating set of $G$, denoted by $\gamma(G)$, is the domination number of $G$. A dominating set $S$ is called a super dominating set of $G$, if for every vertex  $u\in \overline{S}=V-S$, there exists $v\in S$ such that $N(v)\cap \overline{S}=\{u\}$. The cardinality of a smallest super dominating set of $G$, denoted by $\gamma_{sp}(G)$, is the super domination number of $G$. In this paper, we study super domination number of some graph classes and present sharp bounds for some graph operations. 
\end{abstract}


\medskip
\noindent{\bf AMS Subj.\ Class.:} 05C69, 05C76

%%%%%%%%%%%%%%%%%%%%%%%%%%%%%%%%%%%%%%%%%%%%%%%%%%%%%%%%%%%%%%%%%%%%%%%%%%%%%%%%%
%%%%%%%%%%%%%%%%%%%%%%%%%%%%%%%%%%%%%%%%%%%%%%%%%%%%%%%%%%%%%%%%%%%%%%%%%%%%%%%%%
\section{Introduction}
 
Let $G = (V,E)$ be a  graph with vertex set $V$ and edge set $E$. Throughout this paper, we consider graphs without loops and directed edges.
For each vertex $v\in V$, the set $N(v)=\{u\in V | uv \in E\}$ refers to the open neighbourhood of $v$ and the set $N[v]=N(v)\cup \{v\}$ refers to the closed neighbourhood of $v$ in $G$. The degree of $v$, denoted by $\deg(v)$, is the cardinality of $N(v)$.
 A set $S\subseteq V$ is a  dominating set if every vertex in $\overline{S}= V- S$ is adjacent to at least one vertex in $S$.
The  domination number $\gamma(G)$ is the minimum cardinality of a dominating set in $G$. There are various domination numbers in the literature.
For a detailed treatment of domination theory, the reader is referred to \cite{domination}. 

In 2015, Lema\'nska et al. introduced the concept of super domination number \cite{Lemans}. By their definition, a dominating set $S$ is called a super dominating set of $G$, if for every vertex  $u\in \overline{S}$, there exists $v\in S$ such that $N(v)\cap \overline{S}=\{u\}$. The cardinality of a smallest super dominating set of $G$, denoted by $\gamma_{sp}(G)$, is the super domination number of $G$. They obtained super domination number of some specific graphs such as path graphs, cycle graphs, complete graphs and complete bipartite graphs. Also they studied Nordhaus-Gaddum type results for this kind of domination number and found sharp bounds for that. 

Super domination number is studied in literature. In 2016, Krishnakumari et al. found lower bound of super domination number of trees and characterize the trees attaining that bound \cite{Kri}. Later Alfarisi et al. investigated super domination number of unicyclic graphs in 2018 \cite{Alf}. In 2019, Dettlaff et  al. obtained closed formulas and tight bounds for the super dominating number of lexicographic product graphs in terms of invariants of the factor graphs involved in the product \cite{Dett}. Also they showed that the problem of finding the super domination number of a graph is NP-Hard. In the same year, Klein et al. obtained results on the super domination number of corona product graphs and Cartesian product graphs \cite{Kle}. In 2022, Zhuang investigated the ratios between
super domination and other domination parameters in trees \cite{Zhu}.

In this paper, we continue the study of super domination number of a graph. In Section 2, we     investigate super domination number of some  graph classes. In Section 3, we remind the definition of edge removal and contraction and obtain sharp upper and lower bound for super domination number of a graph by this modification. Finally, we mention the definition of  vertex removal and contraction and find sharp bounds for the super domination number of a graph which modified by these operations in Section 4.

\section{Super domination number of some graph classes}

In this section, we study the super domination of some special graph classes. First, we state some known results.

	\begin{theorem}\cite{Lemans}\label{thm-1}
 Let $G$ be a graph of order $n$ which is not empty graph. Then,
 $$1\leq \gamma(G) \leq \frac{n}{2} \leq \gamma_{sp}(G) \leq n-1.$$
	\end{theorem}

	\begin{theorem}\cite{Lemans}\label{thm-2}
	\begin{itemize}
	\item[(i)]
	For a path graph $P_n$ with $n\geq 3$, $\gamma_{sp}(P_n)=\lceil \frac{n}{2} \rceil$.
	\item[(ii)]
	For a cycle graph $C_n$,
	\begin{displaymath}
\gamma_{sp}(C_n)= \left\{ \begin{array}{ll}
\lceil\frac{n}{2}\rceil & \textrm{if $n \equiv 0, 3 ~ (mod ~ 4)$, }\\
\\
\lceil\frac{n+1}{2}\rceil & \textrm{otherwise.}
\end{array} \right.
\end{displaymath}
	\end{itemize}
	\end{theorem}

	\begin{theorem}\cite{Lemans}\label{thm-3}
\begin{itemize}
\item[(i)]
For the complete graph $K_n$, $\gamma_{sp}(K_n)=n-1$.
\item[(ii)]
For the complete bipartite graph $K_{n,m}$, $\gamma_{sp}(K_{n,m})=n+m-2$.
\end{itemize}
	\end{theorem}

Let $n$ be any positive integer and  $F_n$ be friendship graph with
$2n + 1$ vertices and $3n$ edges. In other words, the friendship  graph $F_n$ is a graph that can be constructed by coalescence $n$
copies of the cycle graph $C_3$ of length $3$ with a common vertex. The Friendship Theorem of  Erd\H os,
R\'{e}nyi and S\'{o}s \cite{erdos}, states that graphs with the property that every two vertices have
exactly one neighbour in common are exactly the friendship graphs.
The Figure \ref{friend} shows some examples of friendship graphs. Here we shall investigate the super domination number of friendship graphs.

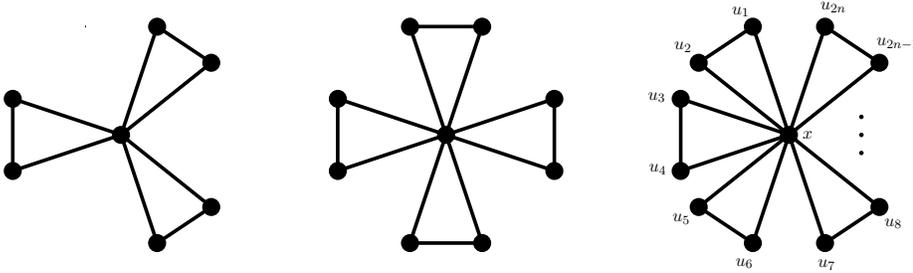
\begin{figure}
\begin{center}
\psscalebox{0.6 0.6}
{
\begin{pspicture}(0,-7.215)(20.277115,-1.245)
\psline[linecolor=black, linewidth=0.04](1.7971154,-1.815)(1.7971154,-1.815)
\psdots[linecolor=black, dotsize=0.4](8.997115,-1.815)
\psdots[linecolor=black, dotsize=0.4](10.5971155,-1.815)
\psdots[linecolor=black, dotsize=0.4](9.797115,-4.215)
\psdots[linecolor=black, dotsize=0.4](8.997115,-6.615)
\psdots[linecolor=black, dotsize=0.4](10.5971155,-6.615)
\psline[linecolor=black, linewidth=0.08](8.997115,-1.815)(10.5971155,-1.815)(8.997115,-6.615)(10.5971155,-6.615)(8.997115,-1.815)(8.997115,-1.815)
\psdots[linecolor=black, dotsize=0.4](12.197115,-3.415)
\psdots[linecolor=black, dotsize=0.4](12.197115,-5.015)
\psdots[linecolor=black, dotsize=0.4](7.397115,-3.415)
\psdots[linecolor=black, dotsize=0.4](7.397115,-5.015)
\psline[linecolor=black, linewidth=0.08](12.197115,-5.015)(7.397115,-3.415)(7.397115,-5.015)(12.197115,-3.415)(12.197115,-5.015)(12.197115,-5.015)
\psdots[linecolor=black, dotsize=0.4](0.1971154,-3.415)
\psdots[linecolor=black, dotsize=0.4](0.1971154,-5.015)
\psdots[linecolor=black, dotsize=0.4](2.5971155,-4.215)
\psline[linecolor=black, linewidth=0.08](2.5971155,-4.215)(0.1971154,-3.415)(0.1971154,-5.015)(2.5971155,-4.215)(2.5971155,-4.215)
\psdots[linecolor=black, dotsize=0.4](3.3971155,-1.815)
\psdots[linecolor=black, dotsize=0.4](4.5971155,-2.615)
\psdots[linecolor=black, dotsize=0.4](3.3971155,-6.615)
\psdots[linecolor=black, dotsize=0.4](4.5971155,-5.815)
\psline[linecolor=black, linewidth=0.08](2.5971155,-4.215)(4.5971155,-5.815)(3.3971155,-6.615)(3.3971155,-6.615)
\psline[linecolor=black, linewidth=0.08](2.5971155,-4.215)(3.3971155,-6.615)(3.3971155,-6.615)
\psline[linecolor=black, linewidth=0.08](2.5971155,-4.215)(3.3971155,-1.815)(4.5971155,-2.615)(2.5971155,-4.215)(2.5971155,-4.215)
\psdots[linecolor=black, dotsize=0.4](15.397116,-2.615)
\psdots[linecolor=black, dotsize=0.4](16.597115,-1.815)
\psdots[linecolor=black, dotsize=0.4](17.397116,-4.215)
\psdots[linecolor=black, dotsize=0.4](15.397116,-5.815)
\psdots[linecolor=black, dotsize=0.4](16.597115,-6.615)
\psdots[linecolor=black, dotsize=0.4](19.397116,-5.815)
\psdots[linecolor=black, dotsize=0.4](18.197115,-6.615)
\psdots[linecolor=black, dotsize=0.4](14.997115,-3.415)
\psdots[linecolor=black, dotsize=0.4](14.997115,-5.015)
\psdots[linecolor=black, dotsize=0.4](18.197115,-1.815)
\psdots[linecolor=black, dotsize=0.4](19.397116,-2.615)
\psdots[linecolor=black, dotsize=0.1](18.997116,-3.815)
\psdots[linecolor=black, dotsize=0.1](18.997116,-4.215)
\psdots[linecolor=black, dotsize=0.1](18.997116,-4.615)
\rput[bl](17.697115,-4.295){$x$}
\rput[bl](16.137115,-1.595){$u_1$}
\rput[bl](14.857116,-2.375){$u_2$}
\psline[linecolor=black, linewidth=0.08](17.397116,-4.215)(19.397116,-2.615)(18.197115,-1.815)(17.397116,-4.215)(16.597115,-1.815)(15.397116,-2.615)(17.397116,-4.215)(14.997115,-3.415)(14.997115,-5.015)(17.397116,-4.215)(15.397116,-5.815)(16.597115,-6.615)(17.397116,-4.215)(18.197115,-6.615)(19.397116,-5.815)(17.397116,-4.215)(17.397116,-4.215)
\rput[bl](14.277116,-3.495){$u_3$}
\rput[bl](14.297115,-5.095){$u_4$}
\rput[bl](14.817116,-6.195){$u_5$}
\rput[bl](16.217115,-7.175){$u_6$}
\rput[bl](18.037115,-7.215){$u_7$}
\rput[bl](19.517115,-6.295){$u_8$}
\rput[bl](18.097115,-1.495){$u_{2n}$}
\rput[bl](19.337116,-2.315){$u_{2n-1}$}
\end{pspicture}
}
\end{center}
\caption{Friendship graphs $F_3$, $F_4$ and $F_n$, respectively.}\label{friend}
\end{figure}

	\begin{theorem}\label{Firend-thm}
 For the friendship graph $F_n$, we have $\gamma_{sp}(F_n)=n+1$.
	\end{theorem}

	\begin{proof}
Consider to the friendship graph $F_n$, as shown in Figure \ref{friend},  with the vertex set $V(F_n)=\{x,u_1,u_2,u_3,u_4,\ldots,u_{2n-1},u_{2n}\}$. Let $D=\{x,u_1,u_3,u_5,\ldots,u_{2n-1}\}$. One can easily check that $D$ is a super dominating set for $F_n$. So $\gamma_{sp}(F_n)\leq n+1$. Now by Theorem \ref{thm-1}, we have $\gamma_{sp}(F_n)\geq n+\frac{1}{2}$. Therefore $\gamma_{sp}(F_n)=n+1$ and we have the result. 
\qed
	\end{proof}

Friendship graphs are a special case of Dutch windmill graphs. The Dutch windmill graph $D_n^{(m)}$ is a graph that can be constructed by coalescence $n$
copies of the cycle graph $C_m$ of length $m$ with a common vertex. Clearly, $D_n^{(3)}=F_n$. Here we state an upper bound for super domination number of these graphs.

		\begin{theorem}\label{Dutch}
For the Dutch windmill graph $D_n^{(m)}$, $m\geq 4$, we have 
$$\gamma_{sp}(D_n^{(m)})\leq n\lceil \frac{m-1}{2} \rceil +1.$$
	\end{theorem}

	\begin{proof}
Suppose that the common vertex is $x$ and the $i$-th copy of cycle graph in $D_n^{(m)}$ is $C_m^{(i)}$. We make a dominating set $D$ for this graph and claim that it is a super dominating set. First, we put $x$ in $D$. Now for every $C_m^{(i)}$, $1\leq i\leq n$, we remove vertex $x$. Therefore we have a path graph of order $m-1\geq 3$. By Theorem \ref{thm-2}, we know that there is a super dominating set for this path of size $\lceil \frac{m-1}{2} \rceil$. Therefore, we put all these vertices in $D$. Clearly $D$ is a super dominating set for $D_n^{(m)}$ with size $n\lceil \frac{m-1}{2} \rceil +1$ and we have the result.
\qed
	\end{proof}

	\begin{remark}\label{remark1}
The upper bound in Theorem \ref{Dutch} is sharp.  It suffices to consider $m\geq 5$ as an odd number. Then $\gamma_{sp}(D_n^{(m)})\leq n \left( \frac{m-1}{2}\right)  +1$. On the other hand, by Theorem \ref{thm-1}, $\gamma_{sp}(D_n^{(m)})\geq \frac{n(m-1)+1}{2}$. So we have $\gamma_{sp}(D_n^{(m)})= n\lceil \frac{m-1}{2} \rceil +1.$
	\end{remark}

As we see in Remark \ref{remark1}, when $m\geq 5$ is an odd number,  then $\gamma_{sp}(D_n^{(m)})= n\lceil \frac{m-1}{2} \rceil +1.$
We end this section by the following example which shows that it is not true for even number $m$ and there are some Dutch windmill graphs $D_n^{(m)}$ which have super domination number less than $n\lceil \frac{m-1}{2} \rceil +1$.

	\begin{example}
 Consider to the graph $D_2^{(8)}$ as shown in Figure \ref{D28}. 
By Theorem \ref{thm-1},  $\gamma_{sp}(D_2^{(8)})\geq \frac{15}{2}$. 
 One can easily check that the set of black vertices is a super dominating set of $D_2^{(8)}$ with size 8. So $\gamma_{sp}(D_2^{(8)})= 8< 2\lceil \frac{8-1}{2} \rceil +1$.
	\end{example}

	\begin{figure}
\begin{center}
\psscalebox{0.6 0.6}
{
\begin{pspicture}(0,-6.0)(10.002778,-3.1972222)
\psdots[linecolor=black, dotsize=0.4](6.201389,-3.398611)
\psdots[linecolor=black, dotsize=0.4](7.4013886,-3.398611)
\psdots[linecolor=black, dotsize=0.4](8.601389,-5.798611)
\psdots[linecolor=black, dotsize=0.4](7.4013886,-5.798611)
\psdots[linecolor=black, dotsize=0.4](1.4013889,-3.398611)
\psdots[linecolor=black, dotsize=0.4](2.601389,-3.398611)
\psdots[linecolor=black, dotsize=0.4](3.8013887,-5.798611)
\psdots[linecolor=black, dotsize=0.4](2.601389,-5.798611)
\psline[linecolor=black, linewidth=0.08](0.20138885,-4.598611)(1.4013889,-3.398611)(3.8013887,-3.398611)(6.201389,-5.798611)(8.601389,-5.798611)(9.801389,-4.598611)(8.601389,-3.398611)(6.201389,-3.398611)(3.8013887,-5.798611)(1.4013889,-5.798611)(0.20138885,-4.598611)(0.20138885,-4.598611)
\psdots[linecolor=black, dotstyle=o, dotsize=0.4, fillcolor=white](5.001389,-4.598611)
\psdots[linecolor=black, dotstyle=o, dotsize=0.4, fillcolor=white](8.601389,-3.398611)
\psdots[linecolor=black, dotstyle=o, dotsize=0.4, fillcolor=white](9.801389,-4.598611)
\psdots[linecolor=black, dotstyle=o, dotsize=0.4, fillcolor=white](6.201389,-5.798611)
\psdots[linecolor=black, dotstyle=o, dotsize=0.4, fillcolor=white](3.8013887,-3.398611)
\psdots[linecolor=black, dotstyle=o, dotsize=0.4, fillcolor=white](0.20138885,-4.598611)
\psdots[linecolor=black, dotstyle=o, dotsize=0.4, fillcolor=white](1.4013889,-5.798611)
\end{pspicture}
}
\end{center}
\caption{Dutch windmill graph $D_2^{(8)}$}\label{D28}
\end{figure}
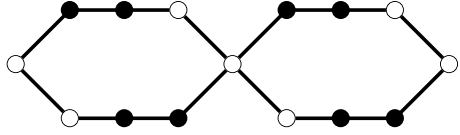

\section{Super domination number of $G-e$ and $G/e$}

The graph $G-e$ is a graph that obtained from $G$ by simply removing the  edge $e$.
In a graph $G$, contraction of an edge $e$ with endpoints $u,v$ is the replacement of $u$ and $v$ with a single vertex such that edges incident to the new vertex are the edges other than $e$ that were incident with $u$ or $v$. The resulting graph $G/e$ has one less edge than $G$ (\cite{Bondy}). We denote this graph by $G/e$. We refer the reader for more results about $G/e$ to \cite{Nima,Nima1}.
In this section we  examine the effects on $\gamma _{sp}(G)$ when $G$ is modified by an edge removal and edge contraction.
 First, we consider to the edge removal of a graph and find upper  and lower bound for super domination number of that.

	\begin{theorem}\label{G-e}
 Let $G=(V,E)$ be a graph and $e=uv\in E$. Then,
 $$\gamma_{sp} (G)-1\leq \gamma_{sp}(G-e)\leq \gamma_{sp} (G)+1.$$
	\end{theorem}

	\begin{proof}
First we find the upper bound for $\gamma_{sp}(G-e)$. Suppose that $D$ is a super dominating set for $G$. We have the following cases:
\begin{itemize}
\item[(i)]
$u,v\in D$. In this case, removing $e$ does not have effect on $D$ and it is a super dominating set for $G-e$ too. So $\gamma_{sp}(G-e)\leq \gamma_{sp} (G)$.
\item[(ii)]
$u,v\notin D$. It is similar to case (i).
\item[(iii)]
$u\in D$ and $v\notin D$. If $u$ is not the vertex such that $N(u)\cap \overline{D} = \{v\}$, then it is similar to case (i). Now suppose that $u$ is the vertex with $N(u)\cap \overline{D} = \{v\}$. Let $D'=D\cup\{v\}$. Clearly $D'$ is a super dominating set for $G-e$. So  $\gamma_{sp}(G-e)\leq \gamma_{sp} (G)+1$.
\end{itemize}
Therefore we have $\gamma_{sp}(G-e)\leq \gamma_{sp} (G)+1$. Now we find a lower bound for $\gamma_{sp}(G-e)$. Suppose that $S$ is a super dominating set for $G-e$. We have the following cases:
\begin{itemize}
\item[(i)]
$u,v\in S$. In this case, $S$ is a super dominating set for $G$ too. So $\gamma_{sp} (G)\leq \gamma_{sp}(G-e)$.
\item[(ii)]
$u,v\notin S$. Then, there are vertices $w,x\in S$ such that $N(w)\cap \overline{S} = \{u\}$ and $N(x)\cap \overline{S} = \{v\}$. Clearly $w \neq x$. So $S$ is a super dominating set for $G$ too, and $\gamma_{sp} (G)\leq \gamma_{sp}(G-e)$.
\item[(iii)]
$u\in S$ and $v\notin S$. In this case, there might be a vertex $s\in \overline{S}$ such that $N(u)\cap \overline{S} = \{s\}$, and $u$ be our only choice. Let $S'=S\cup\{v\}$. Then one can easily check that $S'$ is a super dominating set for $G$ and   $\gamma_{sp} (G)\leq \gamma_{sp}(G-e)+1$.
\end{itemize}
Therefore we have the result.
\qed
	\end{proof}

	\begin{remark}
Bounds in Theorem \ref{G-e} are sharp. For the lower bound, it suffices to consider $G=C_{14}$. Then by Theorem \ref{thm-2},  $\gamma_{sp}(C_{14})=8$ and $\gamma_{sp}(C_{14}-e)=\gamma_{sp}(P_{14})=7$. For the upper bound, consider to the Figure \ref{sharpg-e}. For graph $H$, by Theorem \ref{thm-1}, we know that  $\gamma_{sp}(H)\geq \frac{7}{2}$ and one can easily check that the set of black vertices is a super dominating set for that. So  $\gamma_{sp}(H)=4$. By the same argument, we have $\gamma_{sp}(G)=4$. Now $G-e=H\cup \{v\}$ and by the definition of the super dominating set, $v$ should be in our set and the set of black vertices is a super dominating set for $G-e$. Therefore $\gamma_{sp}(G-e)= \gamma_{sp} (G)+1$.

	\end{remark}

	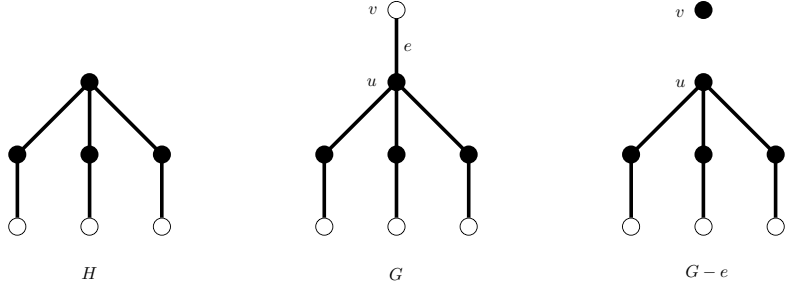
\begin{figure}
	\begin{center}
\psscalebox{0.6 0.6}
{
\begin{pspicture}(0,-6.4893055)(17.202778,-0.30791655)
\psline[linecolor=black, linewidth=0.08](0.20138885,-5.309305)(0.20138885,-3.7093055)(0.20138885,-3.7093055)
\psline[linecolor=black, linewidth=0.08](1.8013889,-5.309305)(1.8013889,-3.7093055)(1.8013889,-2.1093054)(1.8013889,-2.1093054)
\psline[linecolor=black, linewidth=0.08](3.401389,-5.309305)(3.401389,-3.7093055)(1.8013889,-2.1093054)(0.20138885,-3.7093055)(0.20138885,-3.7093055)
\psdots[linecolor=black, dotstyle=o, dotsize=0.4, fillcolor=white](0.20138885,-5.309305)
\psdots[linecolor=black, dotstyle=o, dotsize=0.4, fillcolor=white](1.8013889,-5.309305)
\psdots[linecolor=black, dotstyle=o, dotsize=0.4, fillcolor=white](3.401389,-5.309305)
\psdots[linecolor=black, dotsize=0.4](1.8013889,-2.1093054)
\psdots[linecolor=black, dotsize=0.4](0.20138885,-3.7093055)
\psdots[linecolor=black, dotsize=0.4](1.8013889,-3.7093055)
\psdots[linecolor=black, dotsize=0.4](3.401389,-3.7093055)
\psline[linecolor=black, linewidth=0.08](7.001389,-5.309305)(7.001389,-3.7093055)(7.001389,-3.7093055)
\psline[linecolor=black, linewidth=0.08](8.601389,-5.309305)(8.601389,-3.7093055)(8.601389,-2.1093054)(8.601389,-2.1093054)
\psline[linecolor=black, linewidth=0.08](10.201389,-5.309305)(10.201389,-3.7093055)(8.601389,-2.1093054)(7.001389,-3.7093055)(7.001389,-3.7093055)
\psdots[linecolor=black, dotstyle=o, dotsize=0.4, fillcolor=white](7.001389,-5.309305)
\psdots[linecolor=black, dotstyle=o, dotsize=0.4, fillcolor=white](8.601389,-5.309305)
\psdots[linecolor=black, dotstyle=o, dotsize=0.4, fillcolor=white](10.201389,-5.309305)
\psdots[linecolor=black, dotsize=0.4](8.601389,-2.1093054)
\psdots[linecolor=black, dotsize=0.4](7.001389,-3.7093055)
\psdots[linecolor=black, dotsize=0.4](8.601389,-3.7093055)
\psdots[linecolor=black, dotsize=0.4](10.201389,-3.7093055)
\psline[linecolor=black, linewidth=0.08](13.801389,-5.309305)(13.801389,-3.7093055)(13.801389,-3.7093055)
\psline[linecolor=black, linewidth=0.08](15.401389,-5.309305)(15.401389,-3.7093055)(15.401389,-2.1093054)(15.401389,-2.1093054)
\psline[linecolor=black, linewidth=0.08](17.001389,-5.309305)(17.001389,-3.7093055)(15.401389,-2.1093054)(13.801389,-3.7093055)(13.801389,-3.7093055)
\psdots[linecolor=black, dotstyle=o, dotsize=0.4, fillcolor=white](13.801389,-5.309305)
\psdots[linecolor=black, dotstyle=o, dotsize=0.4, fillcolor=white](15.401389,-5.309305)
\psdots[linecolor=black, dotstyle=o, dotsize=0.4, fillcolor=white](17.001389,-5.309305)
\psdots[linecolor=black, dotsize=0.4](15.401389,-2.1093054)
\psdots[linecolor=black, dotsize=0.4](13.801389,-3.7093055)
\psdots[linecolor=black, dotsize=0.4](15.401389,-3.7093055)
\psdots[linecolor=black, dotsize=0.4](17.001389,-3.7093055)
\psline[linecolor=black, linewidth=0.08](8.601389,-2.1093054)(8.601389,-0.5093054)(8.601389,-0.5093054)
\psdots[linecolor=black, dotsize=0.4](15.401389,-0.5093054)
\psdots[linecolor=black, dotstyle=o, dotsize=0.4, fillcolor=white](8.601389,-0.5093054)
\rput[bl](7.941389,-2.1893053){$u$}
\rput[bl](14.781389,-2.2093055){$u$}
\rput[bl](7.981389,-0.5893054){$v$}
\rput[bl](14.781389,-0.66930544){$v$}
\rput[bl](8.761389,-1.4093055){$e$}
\rput[bl](1.6213889,-6.4693055){$H$}
\rput[bl](8.441389,-6.4893055){$G$}
\rput[bl](15.001389,-6.4693055){$G-e$}
\end{pspicture}
}
\end{center}
\caption{Graphs $H$, $G$ and $G-e$, respectively.}\label{sharpg-e}
\end{figure}
	
Now we consider to the edge contraction of a graph and find upper  and lower bound for super domination number of that.	
	
\begin{theorem}\label{G/e}
Let $G=(V,E)$ be a graph and $e\in E$. Then,
$$\gamma _{sp}(G) - 1 \leq \gamma _{sp}(G/e)\leq \gamma _{sp}(G).$$
\end{theorem}

	\begin{proof}
Suppose that $e=uv\in E$ and  let $w$ be the vertex which is replacement of $u$ and $v$ in $G/e$. First, we find the upper bound for $\gamma _{sp}(G/e)$. Suppose that $D$ is super dominating set of $G$. We consider the following cases:
\begin{itemize}
\item[(i)]
$u,v\in D$. So there may exist vertices $x,y\in \overline{D}$ which are dominated by $u$ and $v$ such that $N(u)\cap \overline{D} = \{x\}$ and $N(v)\cap \overline{D} = \{y\}$. By the definition of super dominating set, $u$ and $v$ can not dominate more than one vertex in $\overline{D}$, so all other neighbours of them are in the dominating set. Now let 
$D'=\left( D- \{u,v\} \right) \cup \{w,x\} $. Since $x\in D'$, any other vertices in $\overline{D}$ except $y$, is dominated by the same vertex as before and $y$ is dominated by $w$ and $N(w)\cap \overline{D} = \{y\}$. Therefore $D'$ is a super dominating set for $G/e$ and $\gamma _{sp}(G/e)\leq \gamma _{sp}(G)$. Also, there may exist only one vertex in $\overline{D}$ like $z$ such that dominates by $u$ or $v$. Suppose that it is $v$. Then $N(v)\cap \overline{D} = \{z\}$. Now let $D''=\left( D- \{u,v\} \right) \cup \{w,z\} $. By the same argument as before, Since $z\in D''$, any other vertices in $\overline{D}$ is dominated by the same vertex as before and $D''$ is a super dominating set for $G/e$ and $\gamma _{sp}(G/e)\leq \gamma _{sp}(G)$. In case no vertices in $\overline{D}$ is dominated by any of $u$ or $v$, then $\left( D- \{u,v\} \right) \cup \{w\} $ is a super dominating set for  $G/e$ since every vertex in $\overline{D}$ is dominated as before, and $\gamma _{sp}(G/e)\leq \gamma _{sp}(G)-1$.
\item[(ii)]
$u\in D$ and $v\notin D$. If $v$ is dominated by $u$ and $N(u)\cap \overline{D} = \{v\}$, then all other neighbours of $u$ should be in $D$. Therefore $\left( D- \{u\} \right) \cup \{w\}$ is a super dominating set for $G/e$ since every vertex in $\overline{D}$ is dominated as before, and $\gamma _{sp}(G/e)\leq \gamma _{sp}(G)$. If $v$ is dominated by a vertex which is not $u$, then $u$ can not dominate any $w\in \overline{D} $, because $ \{v,w\}\subseteq N(u)\cap \overline{D}$. Therefore $\left( D- \{u\} \right) \cup \{w\}$ is a super dominating set for $G/e$ by the same argument as discussed, and $\gamma _{sp}(G/e)\leq \gamma _{sp}(G)$.
\item[(iii)]
$u,v\notin D$. Suppose that $a\in D$ is the vertex that $N(a)\cap \overline{D} = \{u\}$. Then by not changing dominating set for $G/e$, $N(a)\cap \overline{D} = \{w\}$ in $G/e$. So $D$ is a super dominating set for $G/e$, and $\gamma _{sp}(G/e)\leq \gamma _{sp}(G)$.
\end{itemize}
Hence $\gamma _{sp}(G/e)\leq \gamma _{sp}(G)$. Now we find a lower bound for  $\gamma _{sp}(G/e)$. For this purpose, first we form $G/e$ and find a super dominating set for that. Suppose that this set is $S$.  We have the following cases:
\begin{itemize}
\item[(i)]
$w\in S$. We claim that in this case $S'=\left( S- \{w\} \right) \cup \{u,v\}$ is a super dominating set for $G$. If we have $x\in \overline{S}$ such that $N(w)\cap \overline{S} = \{x\}$ in $G/e$, then without loss of generality, suppose that $x\in N(u)$ in $G$. Then, clearly $N(u)\cap \overline{S'} = \{x\}$ in $G$ and any other vertex in $\overline{S'}$ dominated same as before. If we do not have $x\in \overline{S}$ such that $N(w)\cap \overline{S} = \{x\}$ in $G/e$, then all vertices in $\overline{S'}$ dominated same as before. So $\gamma _{sp}(G)\leq \gamma _{sp}(G/e)+1$.
\item[(ii)]
$w\notin S$. Then there exists $y\in S$ such that $N(y)\cap \overline{S} = \{w\}$. We have two cases. First, $y\in N(u)$ and $y\in N(v)$ in $G$. Then $S'=S\cup\{v\}$ is a super dominating set for $G$, because $N(y)\cap \overline{S'} = \{u\}$ and all other vertices in $\overline{S'}$ dominated same as before. Second, $y\in N(u)$ and $y\notin N(v)$ in $G$. By the same argument, $S'=S\cup\{v\}$ is a super dominating set for $G$. Hence $\gamma _{sp}(G)\leq \gamma _{sp}(G/e)+1$.
\end{itemize}
Therefore we have the result.
\qed
	\end{proof}

	\begin{remark}
Bounds in Theorem \ref{G/e} are sharp. For the upper bound, it suffices to consider $G=P_{12}$. Then for every $e\in E(G)$, we have $G/e=P_{11}$. By Theorem \ref{thm-2}, $\gamma _{sp}(P_{12})=\gamma _{sp}(P_{11})=6$. So  $\gamma _{sp}(G/e)= \gamma _{sp}(G)$. For the lower bound, it suffices to consider $H=P_{11}$. Then for every $e\in E(H)$, we have $H/e=P_{10}$. By Theorem \ref{thm-2}, $\gamma _{sp}(P_{10})=5$. Hence  $\gamma _{sp}(H/e)= \gamma _{sp}(H)-1$.
	\end{remark}

We end this section by an immediate result of Theorems \ref{G-e} and \ref{G/e}:

\begin{corollary}
Let $G=(V,E)$ be a graph and $e\in E$. Then,
$$\frac{\gamma _{sp}(G-e)+\gamma _{sp}(G/e)-1}{2}\leq \gamma _{sp}(G) \leq \frac{\gamma _{sp}(G-e)+\gamma _{sp}(G/e)}{2}+1.$$
Let $\alpha = \frac{\gamma _{sp}(G-e)+\gamma _{sp}(G/e)}{2}$. Hence, if $\alpha \in \mathbb{Z}$, then $\gamma _{sp}(G)\in \{\alpha,\alpha+1\}$ and if $\alpha \notin \mathbb{Z}$, then $\gamma _{sp}(G-e)+\gamma _{sp}(G/e)$ is an odd number and $\gamma _{sp}(G)\in \{\alpha-\frac{1}{2},\alpha+\frac{1}{2}\}$. Therefore in general case, 
$$\gamma _{sp}(G)\in \{\alpha-\frac{1}{2},\alpha,\alpha+\frac{1}{2},\alpha+1\},$$
where $\alpha = \frac{\gamma _{sp}(G-e)+\gamma _{sp}(G/e)}{2}$.
\end{corollary}

\section{Super domination number of $G-v$ and $G/v$}

 Let $v$ be a vertex in graph $G$. The graph $G-v$ is a graph that is made  by deleting the vertex $v$ and all edges connected to $v$ from the graph $G$. The contraction of $v$ in $G$ denoted by $G/v$ is the graph obtained by deleting $v$ and putting a clique on the open neighbourhood of $v$. Note that this operation does not create parallel edges; if two neighbours of $v$ are already adjacent, then they remain simply adjacent (see \cite{Walsh}). In this section we  examine the effects on $\gamma _{sp}(G)$ when $G$ is modified by a vertex removal and vertex contraction.
 First, we consider to the vertex removal of a graph and find upper  and lower bound for super domination number of that.

	\begin{theorem}\label{G-v}
 Let $G=(V,E)$ be a graph and $v\in V$. Then,
 $$\gamma_{sp} (G)-1\leq \gamma_{sp}(G-v)\leq \gamma_{sp} (G).$$
	\end{theorem}

	\begin{proof}
 First we find the upper bound for $\gamma_{sp}(G-v)$. Suppose that $D$ is a super dominating set of $G$. We consider to the following cases:
\begin{itemize}

\item[(i)]
$v\notin D$. Then removing $v$ does not have effect on our super dominating set and $D$ is a super dominating set for $G-v$ too, since every vertex in $\overline{D}$ dominated by the same vertex as before. So  $\gamma_{sp}(G-v)\leq \gamma_{sp} (G)$.
\item[(ii)]
$v \in D$. There may not exists any vertex in $\overline{D}$ which is dominated by $v$. So by the same argument as previous case, $D$ is a super dominating set for $G-v$ too, and $\gamma_{sp}(G-v)\leq \gamma_{sp} (G)$.
But, there may exists $u\in \overline{D}$ which is dominated by $v$ such that  $N(v)\cap \overline{D} = \{u\}$. By the definition of super dominating set, $v$ can not dominate more than one vertex in $\overline{D}$. So all other neighbours of $v$ are in $D$. Now let $D'=\left( D-\{v\}\right)\cup \{u\}$. One can easily check that $D'$ is a super dominating set for $G-v$. Hence $\gamma_{sp}(G-v)\leq \gamma_{sp} (G)$.
\end{itemize}
Therefore in all cases, $\gamma_{sp}(G-v)\leq \gamma_{sp} (G)$. Now we find a lower bound for  $\gamma_{sp}(G-v)$. First we form $G-v$ and find a super dominating set for that. Suppose that this set is $S$. Now let $S'=S\cup \{v\}$ and add all removed edges connected to $v$ to form $G$. Then all vertices in $\overline{S'}$ dominated by the same vertex as they dominated by vertices in $S$. Hence  $\gamma_{sp} (G)\leq \gamma_{sp}(G-v)+1$ and therefore we have the result. 
\qed
	\end{proof}

	\begin{remark}
Bounds in Theorem \ref{G-v} are sharp. For the upper bound, it suffices to consider $G=P_5$ and $v$ is adjacent to a pendant vertex. Then one can easily check that $\gamma_{sp}(G-v)=\gamma_{sp} (G)$. For the lower bound, it suffices to consider  $H=P_5$ and  $v$ is a pendant vertex. Then by Theorem \ref{thm-2}, $\gamma_{sp} (H)= 3$ and $\gamma_{sp} (H-v)= 2$ and therefore  $\gamma_{sp} (H)-1= \gamma_{sp}(H-v)$.
	\end{remark}

Before we  examine the effects on $\gamma _{sp}(G)$ when $G$ is modified by a vertex contraction, we show that there are some graphs $H=(V(H),E(H))$ and $v\in V(H)$ such that $|\gamma_{sp}(H)- \gamma_{sp} (H/v)|$ can be arbitrary large.

	\begin{theorem}
For every $n \in \mathbb{N}$, there exists a graph $G=(V,E)$ and a vertex $v\in V$ such that $\gamma_{sp}(G/v) - \gamma_{sp} (G)=n$.
	\end{theorem}

	\begin{proof}
Let $G=F_{n+2}$ be a friendship graph of order $n+2$ (see $F_n$ in Figure \ref{friend}), and $v$ be the common vertex with degree $2n+4$. By Theorem \ref{Firend-thm}, $\gamma_{sp} (F_{n+2})=n+3$. Now by removing $v$ and putting a clique on the open neighbourhood of that, we have $G/v=K_{2n+4}$. By Theorem \ref{thm-3}, $\gamma_{sp} (K_{2n+4})=2n+3$. Therefore we have the result.
\qed
	\end{proof}	

Now we find upper  and lower bound for super domination number of  a graph when it is modified by  vertex contraction. First we consider to pendant vertices.

	\begin{theorem}\label{G/vpendant}
  Let $G=(V,E)$ be a graph and $v\in V$ is a pendant vertex. Then,
 $$\gamma_{sp} (G)-1\leq \gamma_{sp}(G/v)\leq \gamma_{sp} (G).$$
	\end{theorem}

	\begin{proof}
If $v$ be a pendant vertex, then $G/v=G-v$. Therefore by Theorem \ref{G-v}, we have the result.
\qed
	\end{proof}

	\begin{remark}
Bounds in Theorem \ref{G/vpendant} are sharp. For the upper bound, it suffices to consider $G=P_6$ and $v$ is  a pendant vertex. Then by Theorem \ref{thm-2}, $\gamma_{sp}(G/v)=\gamma_{sp} (G)=3$.  For the lower bound, it suffices to consider $G=K_{1,n}$ and $v$ is  a pendant vertex. So $K_{1,n}/v=K_{1,n-1}$. Then by Theorem \ref{thm-3}, $\gamma_{sp}(G/v)=\gamma_{sp} (G)-1=n-2$.
	\end{remark}

As an an immediate result of Theorems \ref{G-v} and \ref{G/vpendant} we have:

\begin{corollary}
Let $G=(V,E)$ be a graph and $v\in V$ is a pendant vertex. Then,
$$\frac{\gamma _{sp}(G-v)+\gamma _{sp}(G/v)}{2}\leq \gamma _{sp}(G) \leq \frac{\gamma _{sp}(G-v)+\gamma _{sp}(G/v)}{2}+1.$$
Let $\beta = \frac{\gamma _{sp}(G-v)+\gamma _{sp}(G/v)}{2}$. Hence, if $\beta \in \mathbb{Z}$, then $\gamma _{sp}(G)\in \{\beta,\beta+1\}$ and if $\beta \notin \mathbb{Z}$, then $\gamma _{sp}(G-v)+\gamma _{sp}(G/v)$ is an odd number and $\gamma _{sp}(G)= \beta+\frac{1}{2}$.
\end{corollary}

Now we consider to vertex $v$ of a graph $G$ which is not pendant and find upper and lower bound for $\gamma_{sp}(G/v)$.

	\begin{theorem}\label{G/v}
 Let $G=(V,E)$ be a graph and $v\in V$ is not a pendant vertex.  Then,
 $$\gamma_{sp} (G)-1\leq \gamma_{sp}(G/v)\leq \gamma_{sp} (G)+\lfloor \frac{\deg (v)}{2} \rfloor -1.$$
	\end{theorem}

	\begin{proof}
Suppose that $v\in V$ such that  $\deg (v)=n\geq2$ and $N(v)=\{v_1,v_2,\ldots,v_n\}$.  First we find an upper bound for $\gamma_{sp}(G/v)$. Suppose that $D$ is a super dominating set for $G$. We have the following cases:
\begin{itemize}
\item[(i)]
$v\notin D$. So, there exists $v_r\in N(v)\cap D$ such that $N(v_r)\cap \overline{D} = \{v\}$ which means that all other neighbours of $v_r$ are in $D$ too. There is no vertex such as $v_p\in N(v)$  that dominates $v_q\in N(v)\cap \overline{D}$ and satisfies the condition of super dominating set, because in that case we have  $\{v_q,v\}\subseteq N(v_p)\cap \overline{D} $ which is a contradiction. So all vertices in $ N(v)\cap \overline{D}$ dominate by some vertices which are not in $N(v)$.  Now by removing $v$ and putting a clique on the open neighbourhood of that, $D$ is a super dominating set for the new graph too. Hence, $\gamma_{sp}(G/v)\leq \gamma_{sp} (G)$.
\item[(ii)]
$v\in D$ and for some $1 \leq i \leq n$, there exists $v_i\in N(v)$ such that  $N(v)\cap \overline{D} = \{v_i\}$. So, all other neighbours of $v$ should be in $D$. Now let $D'=\left(D-\{v\} \right)\cup \{v_i\} $. Clearly, $D'$ is a super dominating set for $G/v$, since every vertex in $\overline{D'}$ is dominated by the same vertex as before. So, in this case $\gamma_{sp}(G/v)\leq \gamma_{sp} (G)$.
\item[(iii)]
$v\in D$ and for every $1 \leq i \leq n$, there does not exist $v_i\in N(v)$ such that  $N(v)\cap \overline{D} = \{v_i\}$. If $v_i\in N(v)$ is dominated by $v_i'$ such that $v_i'\notin N(v)$, then after removing $v$ and putting a clique on the open neighbourhood of that, $v_i$ still can dominated by $v_i'$ and $N(v_i')\cap \overline{D} = \{v_i\}$. So we keep all vertices in $D-N(v)$ in our dominating set. If $v_i\in N(v)$ is dominated by $v_j$ such that $v_j\in N(v)$ and $N(v_j)\cap \overline{D} = \{v_i\}$, Then we simply add $v_i$ to our dominating set. At most we have $\lfloor \frac{n}{2} \rfloor$ vertices with this condition. Without loss of generality, suppose that $v_1$ dominates $v_2$, $v_3$ dominates $v_4$ and so on. Then by our argument, $(D-\{v\})\cup\{v_2,v_4,\ldots\}$ is a super dominating set for $G/v$ and $\gamma_{sp}(G/v)\leq \gamma_{sp} (G)+\lfloor \frac{n}{2} \rfloor -1$. Note that if $\deg (v)$ is an odd number, then one of the vertices in the neighbourhood of $v$, say $x$, may not be adjacent to any vertex in $N(v)$, but it will be adjacent to all of them in $G/v$. Since all vertices in $N(v)-\{x\}$ are in our new dominating set, then $x$ is dominated by the same vertex as before or is in the dominating set which both cases do not have effect on our dominating set. 
\end{itemize}
So by our argument $\gamma_{sp}(G/v)\leq \gamma_{sp} (G)+\lfloor \frac{\deg (v)}{2} \rfloor -1$. Now we find a lower bound for $\gamma_{sp}(G/v)$. First we remove $v$ and put a clique on the open neighbourhood of that. Now we find a super dominating set dominating set for $G/v$. Suppose that this set is $S$. If we have two or more vertices in $N(v)$ such that they are not in $S$, then these vertices can not dominate by vertices in $N(v)$ because it contradicts with the definition of super dominating set. Also by the same reason, all of the vertices in $G$ dominate by some vertices which are not adjacent to $v$.  So all vertices in $\overline{S}$ are dominated by vertices in $S-N(v)$. Hence if we form $G$ and add and remove all the corresponding edges, then $S\cup\{v\}$ is a super dominating set for $G$. If we have only one vertex in $N(v)$, say $x$, such that  $x\notin S$, then when we form $G$, the set $S\cup\{v\}$ is a super dominating set for $G$. Because $x$ is dominated by $v$ and $N(v)\cap \overline{S} = \{x\}$, and the rest of vertices in $\overline{S}$, are dominated by the same vertices as before. In case all vertices in $N(v)$ be in $S$, Then by the same argument, again $S\cup\{v\}$ is a super dominating set for $G$. Therefore $\gamma_{sp}(G)\leq \gamma_{sp} (G/v)+1$ and we have the result.
\qed
	\end{proof}

	\begin{remark}
Bounds in Theorem \ref{G/v} are sharp. For the upper bound, it suffices to consider $G$ as shown in Figure \ref{G/vnonpendant}. One can easily check that the set of black vertices in $G$ and $G/v$ are super dominating set for them. Now we show that we can not have super dominating set for them with smaller size. For $G$, by Theorem \ref{thm-1}, $\gamma_{sp} (G)\geq \frac{9}{2}$. Hence $\gamma_{sp} (G)=5$. For $G/v$, Fisrt we consider to all vertices except $y$. By the definition of super dominating set, at least six of them should be in our dominating set. Among $x$ and $y$, we need at least one them in our dominating set too. So we need at least seven vertices in our dominating set. Hence $\gamma_{sp} (G/v)=7$. Therefore $\gamma_{sp}(G/v)= \gamma_{sp} (G)+\lfloor \frac{\deg (v)}{2} \rfloor -1$. For the lower bound, it suffices to consider  $H=P_5$ and  $v$ is not a pendant vertex. So $P_5/v=P_4$. Then by Theorem \ref{thm-2}, $\gamma_{sp} (H)= 3$ and $\gamma_{sp} (H/v)= 2$ and therefore  $\gamma_{sp} (H)-1= \gamma_{sp}(H/v)$.
	\end{remark}	
	
\begin{figure}
\begin{center}
\psscalebox{0.6 0.6}
{
\begin{pspicture}(0,-7.1014423)(14.607116,0.29567322)
\psline[linecolor=black, linewidth=0.04](0.01,-1.1014422)(0.01,-1.1014422)
\psdots[linecolor=black, dotsize=0.4](3.21,-1.5014423)
\psdots[linecolor=black, dotsize=0.4](3.21,0.09855774)
\psdots[linecolor=black, dotsize=0.4](3.21,-3.5014422)
\psdots[linecolor=black, dotsize=0.4](2.41,-5.901442)
\psdots[linecolor=black, dotsize=0.4](4.01,-5.901442)
\psdots[linecolor=black, dotsize=0.4](5.61,-2.7014422)
\psdots[linecolor=black, dotsize=0.4](5.61,-4.301442)
\psdots[linecolor=black, dotsize=0.4](0.81,-2.7014422)
\psdots[linecolor=black, dotsize=0.4](0.81,-4.301442)
\psline[linecolor=black, linewidth=0.08](5.61,-4.301442)(0.81,-2.7014422)(0.81,-4.301442)(5.61,-2.7014422)(5.61,-4.301442)(5.61,-4.301442)
\psline[linecolor=black, linewidth=0.08](3.21,0.09855774)(3.21,-1.5014423)(3.21,-3.5014422)(4.01,-5.901442)(2.41,-5.901442)(3.21,-3.5014422)(3.21,-3.5014422)
\rput[bl](3.31,-3.2214422){$v$}
\rput[bl](3.53,-1.5814422){$x$}
\rput[bl](3.55,-0.02144226){$y$}
\psdots[linecolor=black, dotsize=0.4](12.01,-1.5014423)
\psdots[linecolor=black, dotsize=0.4](12.01,0.09855774)
\psdots[linecolor=black, dotsize=0.4](11.21,-5.901442)
\psdots[linecolor=black, dotsize=0.4](12.81,-5.901442)
\psdots[linecolor=black, dotsize=0.4](14.41,-2.7014422)
\psdots[linecolor=black, dotsize=0.4](14.41,-4.301442)
\psdots[linecolor=black, dotsize=0.4](9.61,-2.7014422)
\psdots[linecolor=black, dotsize=0.4](9.61,-4.301442)
\rput[bl](12.33,-1.5814422){$x$}
\rput[bl](12.35,-0.02144226){$y$}
\rput[bl](3.01,-7.0214424){$G$}
\rput[bl](11.57,-7.1014423){$G/v$}
\psline[linecolor=black, linewidth=0.08](12.01,0.09855774)(12.01,-1.5014423)(14.41,-2.7014422)(14.41,-4.301442)(12.81,-5.901442)(11.21,-5.901442)(9.61,-4.301442)(9.61,-2.7014422)(12.01,-1.5014423)(11.21,-5.901442)(11.21,-5.901442)
\psline[linecolor=black, linewidth=0.08](12.01,-1.5014423)(12.81,-5.901442)(12.81,-5.901442)
\psline[linecolor=black, linewidth=0.08](9.61,-4.301442)(12.01,-1.5014423)(14.41,-4.301442)(9.61,-4.301442)(14.41,-2.7014422)(12.81,-5.901442)(9.61,-2.7014422)(14.41,-2.7014422)(11.21,-5.901442)(9.61,-2.7014422)(9.61,-2.7014422)
\psline[linecolor=black, linewidth=0.08](9.61,-2.7014422)(14.41,-4.301442)(14.41,-4.301442)
\psline[linecolor=black, linewidth=0.08](11.21,-5.901442)(14.41,-4.301442)(14.41,-4.301442)
\psline[linecolor=black, linewidth=0.08](12.81,-5.901442)(9.61,-4.301442)(9.61,-4.301442)
\psdots[linecolor=black, dotstyle=o, dotsize=0.4, fillcolor=white](3.21,-1.5014423)
\psdots[linecolor=black, dotstyle=o, dotsize=0.4, fillcolor=white](12.01,-1.5014423)
\psdots[linecolor=black, dotstyle=o, dotsize=0.4, fillcolor=white](0.81,-2.7014422)
\psdots[linecolor=black, dotstyle=o, dotsize=0.4, fillcolor=white](2.41,-5.901442)
\psdots[linecolor=black, dotstyle=o, dotsize=0.4, fillcolor=white](5.61,-2.7014422)
\end{pspicture}
}
\end{center}
\caption{Graphs $G$ and $G/v$}\label{G/vnonpendant}
\end{figure}
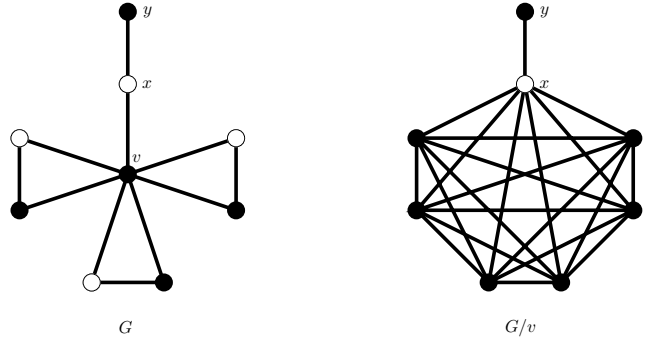

We end this paper by an immediate result of Theorems \ref{G-v} and \ref{G/v}.

\begin{corollary}
Let $G=(V,E)$ be a graph and $v\in V$ is not a pendant vertex. Then,
$$\frac{\gamma _{sp}(G-v)+\gamma _{sp}(G/v)-\lfloor \frac{\deg (v)}{2} \rfloor +1}{2}\leq \gamma _{sp}(G) \leq \frac{\gamma _{sp}(G-v)+\gamma _{sp}(G/v)}{2}+1.$$
Let $\theta = \frac{\gamma _{sp}(G-v)+\gamma _{sp}(G/v)}{2}$ and $\deg(v)=2$. Hence, if $\theta \in \mathbb{Z}$, then $\gamma _{sp}(G)\in \{\theta,\theta +1\}$ and if $\theta \notin \mathbb{Z}$, then $\gamma _{sp}(G-v)+\gamma _{sp}(G/v)$ is an odd number and $\gamma _{sp}(G)\in \{\theta-\frac{1}{2} , \theta+\frac{1}{2} \}$. Therefore in general case, if $\deg(v)=2$, then
$$\gamma _{sp}(G)\in \{\theta-\frac{1}{2},\theta,\theta+\frac{1}{2},\theta+1\},$$
where $\theta = \frac{\gamma _{sp}(G-v)+\gamma _{sp}(G/v)}{2}$.
\end{corollary}

\section{Conclusions}

In this paper, we obtained some lower and upper bounds of super domination number of graphs which modified  by vertex and edge removing, and also vertex and edge contraction regarding the super domination number of the main graph.  Also
we showed that these bounds are sharp.We presented some results for super domination number of some specific graph clasess. Future topics of interest for future research include the following suggestions:

\begin{itemize}
\item[(i)]
Finding super domination number of Dutch windmill graph $D_n^{(m)}$ when $m$ is an even number.
\item[(ii)]
Finding super domination number of subdivision and power of a graph.
\item[(iii)]
Finding super  domination number of link, chain, etc of graphs.
\end{itemize}

	\section{Acknowledgements} 

	The  author would like to thank the Research Council of Norway and Department of Informatics, University of Bergen for their support.

\end{document}